\pdfoutput=1
\RequirePackage{ifpdf}
\ifpdf 
\documentclass[pdftex]{sigma}
\else
\documentclass{sigma}
\fi

\usepackage{dsfont}

\usepackage{tikz}
\usepackage{tikz-3dplot-circleofsphere}
\usetikzlibrary{shapes,arrows}
\usetikzlibrary{shapes, arrows,positioning,graphs,babel,quotes,fit, backgrounds,scopes,chains,calc}
\usetikzlibrary{arrows.meta}

\numberwithin{equation}{section}

\newtheorem{them}{Theorem}[section]
\newtheorem{con}{Conjecture}[section]
\newtheorem{cor}[them]{Corollary}

\newtheorem{pro}[them]{Proposition}

\theoremstyle{definition}
\newtheorem{defi}[them]{Definition}

\newtheorem{rem}[them]{Remark}

\newtheorem{ex}[them]{Example}

\newcommand{\X}{\mathcal{X}}

\newcommand{\C}{\mathbb{C}}
\newcommand{\R}{\mathbb{R}}
\newcommand{\Sp}{\mathbb{S}}
\newcommand{\N}{\mathbb{N}}
\newcommand{\Z}{\mathbb{Z}}

\renewcommand{\P}{\mathbb{P}}

\renewcommand{\X}{\mathfrak{X}}

\newcommand{\p}{\mathcal{P}}
\newcommand{\M}{\mathcal{M}}
\renewcommand{\d}{\mathrm{d}}

\newcommand{\1}{\mathds{1}}
\newcommand{\eps}{\varepsilon}

\begin{document}

\allowdisplaybreaks

\renewcommand{\thefootnote}{}

\renewcommand{\PaperNumber}{009}

\FirstPageHeading

\ShortArticleName{Taking Music Seriously}

\ArticleName{Taking Music Seriously: on the Dynamics\\ of `Mathemusical' Research with a Focus\\ on Hexachordal Theorems\footnote{This paper is a~contribution to the Special Issue on Differential Geometry Inspired by Mathematical Physics in honor of Jean-Pierre Bourguignon for his 75th birthday. The~full collection is available at \href{https://www.emis.de/journals/SIGMA/Bourguignon.html}{https://www.emis.de/journals/SIGMA/Bourguignon.html}}}

\Author{Moreno ANDREATTA~$^{\rm a}$, Corentin GUICHAOUA~$^{\rm b}$ and Nicolas JUILLET~$^{\rm c}$}

\AuthorNameForHeading{M.~Andreatta, C.~Guichaoua and N.~Juillet}

\Address{$^{\rm a)}$~IRMA-CNRS-CREAA-University of Strasbourg and IRCAM, Paris, France}
\EmailD{\href{mailto:andreatta@math.unistra.fr}{andreatta@math.unistra.fr}}
\URLaddressD{\url{http://repmus.ircam.fr/moreno-en/home}}

\Address{$^{\rm b)}$~Independent Researcher, SMIR Project, France}
\EmailD{\href{mailto:corentin.guichaoua@gmail.com}{corentin.guichaoua@gmail.com}}

\Address{$^{\rm c)}$~IRIMAS, Universit\'e de Haute-Alsace, France}
\EmailD{\href{mailto:nicolas.juillet@uha.fr}{nicolas.juillet@uha.fr}}
\URLaddressD{\url{https://juillet.perso.math.cnrs.fr/}}

\ArticleDates{Received July 01, 2023, in final form January 11, 2024; Published online January 25, 2024}

\Abstract{After presenting the general framework of `mathemusical' dynamics, we focus on one music-theoretical problem concerning a special case of homometry theory applied to music composition, namely Milton Babbitt's hexachordal theorem. We briefly discuss some historical aspects of homometric structures and their ramifications in crystallography, spectral analysis and music composition via the construction of rhythmic canons tiling the integer line. We then present the probabilistic generalization of Babbitt's result we recently introduced in a paper entitled ``New hexachordal theorems in metric spaces with probability measure'' and illustrate the new approach with original constructions and examples.}


\Classification{00A65; 28A75; 05C12; 60D05}

\renewcommand{\thefootnote}{\arabic{footnote}}
\setcounter{footnote}{0}

\section{An introduction to ``mathemusical'' research}

Despite a very long historical relationship between mathematics and music, which went through the centuries from Pythagoras to nowadays \cite{Fauvel2006}, the real interest of the community of ``working mathematicians'' in this research field is a relatively recent phenomenon. One may surely find the germs of a new way of looking at the relations between music and (modern) mathematics in the second half of the twentieth century, with some remarkable figures of mathematically-inclined music-theorists and composers such as Iannis Xenakis (in Europe) and Milton Babbitt (in the United States). Both played an important role in the emergence of an \textit{algebraic} approach in the formalization of twentieth-century music theory and composition, which constitutes a~crucial moment in the development of a systematic orientation in contemporary musicology, both in the European and in the American tradition \cite{Andreatta2008}. As widely discussed in another contribution of this special issue,\footnote{See the eleven variations on the maths and music theme by Fran\c{c}ois Nicolas, published in the present volume.} Jean-Pierre Bourguignon has undoubtedly given a major contribution in the progressive institutionalization process of this research domain. He was in fact one of the main catalyzers of the Diderot forum on mathematics and music, which took place in 1999 simultaneously in Paris, Vienna and Lisbon and that was organized under the auspices of the European Mathematical Society. The conference, as well as the volume that was published subsequently by Springer \cite{Assayag}, represented a real milestone in the change of perspective by mathematicians on music and mathematics as a truly interdisciplinary research field. This led in 2007 to the constitution of an international society (the ``Society for Mathematics and Computation in Music'')\footnote{See \url{http://www.smcm-net.info}.} and the launch of the first mathematical journal devoted to maths and music research (\textit{J.~Math. Music}, edited by Taylor and Francis).\footnote{See \url{https://www.tandfonline.com/journals/tmam20}.} The recognition of the relevant mathematical dimension of the research carried on in this domain enabled the inscription in 2010 of ``Mathematics and Music'' as an official topic within the mathematics subject classification under the code~00A65.

\subsection{Around the domain of structural music information research}

In contrast to statistical methods and signal-based approaches currently employed in different domains of computational musicology and music information retrieval,\footnote{See \cite{Meredith2016} for a survey of the different approaches in computational music analysis. The reader may refer to~\cite{Peeters2013} for a roadmap in the domain of music information research in both symbolic and signal-based approaches and edited by the MIReS (Music Information Research) Consortium.} what we originally suggested to call ``mathemusical'' research \cite{Andreatta2003} shows the interest of introducing a structural perspective into the multidisciplinary field of music information research making use of advanced mathematics. The research presented in this paper is carried on, in particular, within the ongoing SMIR Project devoted to structural modern mathematics applied to the field of music information research. This project was initially supported by the University of Strasbourg Institute for Advanced Study (USIAS) and is currently hosted as a permanent transversal axis at the \textit{Institut de Recherche Math\'ematique Avanc\'ee} (IRMA) in collaboration with the CREAA (\textit{Centre de Recherche et d'Exp\'erimentation sur l'Acte Artistique}) and the \textit{Institut de Recherche et Coordination Acoustique/Musique} (IRCAM) in Paris.\footnote{For a detailed description of the SMIR project, together with the list of participants and scientific production, including academic work at Master and PhD levels, see \url{http://repmus.ircam.fr/moreno/smir}.}

``Mathemusical'' research, as carried on within the SMIR Project, is based on the interplay between several mathematical disciplines: algebra, combinatorics, geometry, topology, category theory, statistics and probability theory. It opens promising perspectives on important prevailing challenges in the domain of relations between music and mathematics, such as the automatic classification of musical styles or the solution of open mathematical conjectures. It therefore asks for new collaborations between mathematicians, computer scientists, musicologists, and composers.

As we suggest with the present article, music surely deserves to be taken seriously by mathematicians since it provides a number of difficult theoretical problems, some of which can even be the starting point for approaching, in a new way, open mathematical conjectures. Among the music-theoretical problems showing a remarkable link with interesting mathematical constructions and sometimes open conjectures, one may quote the following ones that still constitute active research axes in the maths and music domain. According to the idea of ``mathemusical'' dynamics we first indicate in the list the original musical problem followed by the mathematical theory in which such a problem can be formalized:\footnote{Obviously, the list does not entirely cover the domain of music-theoretical problems that can be interesting to approach from a mathematical perspective. The reader may find some more examples in \textit{J.~Math.\ Music} as well as in the \textit{Proceedings} of the Mathematics and Computation in Music Conferences that have been regularly edited by Springer.}
\begin{itemize}\itemsep=0pt
\item Tiling rhythmic canons and their Fourier-based characterization (with a connection to Fuglede's spectral conjecture). See \cite{Amiot2016, Lanzarotto2022}.
\item Z-relation in music theory as connected to the study of homometric structures in crystallography (and their extensions via the notion of $k$-deck). See \cite{mandereau2011b,mandereau2011a};
\item Transformational music theory and the functorial representation of poly-Klumpenhouwer networks as categorical graph-theoretical constructions. See \cite{Popoff2018}.
\item Neo-Riemannian music analysis, spatial computing and lattice-based representations of musical structures (with tools derived from \textit{formal concept analysis} and \textit{mathematical morphology}). See \cite{Freund2015, Lascabettes2022}.
\item Diatonic theory, maximally even sets and the discrete Fourier transform \cite{AmiotJMM}.
\item Periodic musical sequences and finite difference calculus taking values in generic finite groups \cite{Andreatta2001}.
\item Chord and rhythms classification in music composition expressed in terms of combinatorial block-designs \cite{Jedrzejewski2009}.
\item Voice-leading theory and the geometry of orbifolds \cite{Tymoczko2011}.
\end{itemize}

As a pedagogical illustration of the ``mathemusical'' dynamics, we will briefly present in the final part of this introductory section the two first research topics from the previous list, insisting on their mutual relations. Interestingly, these two topics, as most of the problems listed before, are in fact deeply interrelated which shows the existence of a remarkable interplay between algebraic formalization and geometric representations of musical structures and processes. From a more philosophical perspective, this interplay provides a further example of the duality between ``temporal'' and ``spatial'' constructions which are surely two of the most fundamental ingredients of music. To quote, for example, Alain Connes' opinion, as expressed in a conversation with composer Pierre Boulez devoted to creativity in mathematics and music: ``Concerning music, it takes place in time, like algebra. In mathematics, there is this fundamental duality between, on the one hand, geometry -- which corresponds to the visual arts, an immediate intuition -- and on the other hand algebra. This is not visual, it has a temporality. This fits in time, it is a computation, something that is very close to the language, and which has its diabolical precision. [\dots] And one only perceives the development of algebra through music''~\cite{Connes}.

The crucial point is the existence of a permanent feedback loop between musical thought, mathematical formalisation and computational modeling, which constitutes the heart of a~new kind of dynamics between music and mathematics via computer-science. This dynamic back-and-forth is illustrated in the diagram of Figure~\ref{fig:diagram1} which clearly shows the place of the three main disciplines (music, computer science and mathematics) as well as the contribution of two additional fields within contemporary ``mathemusical'' research, namely epistemology and cognition.\footnote{See \cite{Andreatta2014} for a more philosophical account of the research in music and mathematics also including an epistemological discussion and some reflexions about the role of cognitive science with respect to mathemusical dynamics.}

\begin{figure}[t]\centering
\includegraphics[width=5cm]{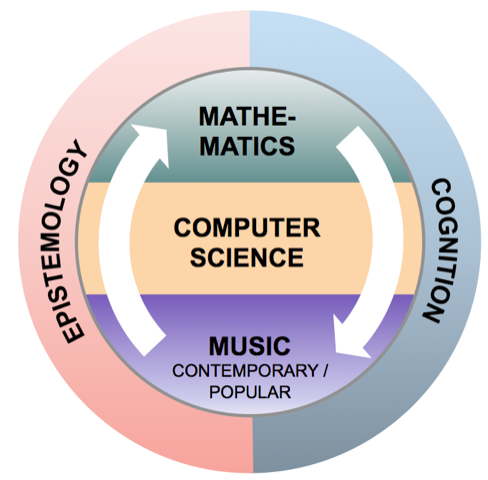}
\caption{A diagram showing the underlying ``mathemusical'' dynamics between music and mathematics through computer-science, also including some epistemological and cognitive aspects.}\label{fig:diagram1}
\end{figure}

Moreover, the mathemusical dynamics (from music to mathematics to music via computer science and the possible epistemological and cognitive implications) constitutes a radical change of perspective with respect to the traditional application of mathematics in the musical domain. Mathemusical problems are characterized by the fact that settling them in an appropriate mathematical framework not only gives rise to new musical applications, but also paves the way to new mathematical constructions. By carefully analyzing the various steps of this mathemusical dynamics, one observes that it can be decomposed into the following three stages (see Figure~\ref{fig:diagram2}):
\begin{itemize}\itemsep=0pt
\item \textit{Formalization}: the initial music-theoretical problem is approached by means of a combination of mathematical tools enabling its formalization and revealing its computational character;
\item\textit{Generalization}: the formalized problem is generalized by using a panoply of mathematical constructions, ranging from abstract algebra to topology and category theory and leading to general statements (or theorems);
\item \textit{Application}: once a generalized result has been obtained, it can be applied to music by focusing on one of the three main aspects, i.e., the theoretical, the analytical and the compositional one.
\end{itemize}

This decomposition of the mathemusical dynamics, together with the triple perspective of the possible musical applications of a general result is shown in Figure~\ref{fig:diagram2}.

\begin{figure}[h!]\centering
\includegraphics[width=8cm]{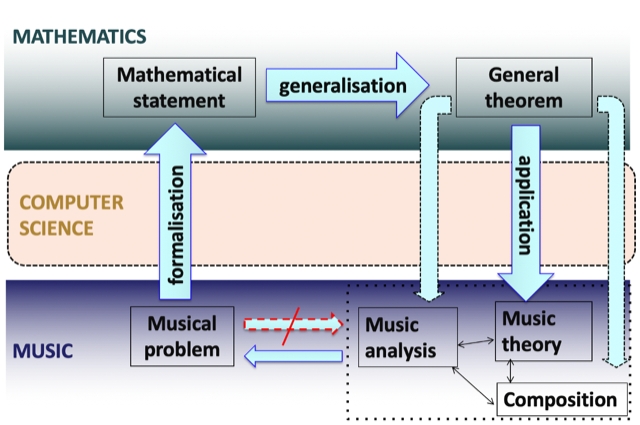}
\caption{A more detailed perspective on the ``mathemusical'' research diagram as presented in Figure~\ref{fig:diagram1}, with the indication of the three main ingredients of the dynamics (namely ``formalization'', ``generalization'' and ``application'').}\label{fig:diagram2}
\end{figure}

It is this fruitful double movement, from music to mathematics and backwards, which is at the heart of a research activity, where computer science is positioned in the middle of this feedback loop, as an interface for connecting the musical and mathematical domains. We simplify the picture by considering a homogeneous intermediate level corresponding to the place occupied by computer science with respect to music theoretical and mathematical research. Through a more careful analysis of the different music theoretical problems, one may nevertheless distinguish the cases in which the computer-aided models are directly built in the formalization process such as problems asking for a computational exploration of the solution space, more than a~search of a~general underlying mathematical theory. To this family belong, for example, typical enumeration problems such as the classification of all possible Hamiltonian paths and cycles in some graph-theoretical musical spaces, such as the Tonnetz and their multi-dimensional extensions~\cite{Andreatta-Baroin2016}. Conversely, there are cases in which the computational models are built starting from some general algebraic results, as in the case of the construction of tiling rhythmic canons that we briefly recall by stressing their link with two fields of mathematics: homometry and spectral theory.

\subsection{Approaching spectral theory and homometry musically}
Applying advanced mathematics to the field of computational musicology is necessary, we believe, to successfully tackle difficult ``mathemusical'' problems which are linked to open conjectures in mathematics. This is the case of two major problems that have been the object of study in the last fifteen years and which can be approached in a new way by stressing their mutual theoretical and computational interplay: the construction of tiling rhythmic canons and the classification of homometric musical structures.

Tiling canons are special rhythmic canons having the property of tiling the time axis by temporal translation of a given rhythmic pattern. Algebraically, it corresponds to the decomposition of a cyclic group of order \textit{n} into a direct sum of two subsets (which can be or not periodic). See~\cite{Andreatta2009} and \cite{Amiot2011} for two journal special issues devoted to the theoretical and computational aspects of the construction of tiling rhythmic canons.

Concerning homometry theory, it is a field in mathematical combinatorics that originates in crystallography, where as originally shown by Arthur Lindo Patterson \cite{patterson39, patterson44} one may find crystals having the same X-ray spectrum without being the same. In other words, these crystal structures are undistinguishable by spectroscopic analysis although the is no isometric group transforming one structure into the other one. Analogously, music composition naturally provides examples of homometric structures in different domains such as harmony, melody or rhythm. By definition, two musical structures are \textit{homometric} if they share the same multiplicity of occurrence of distance between their elements. Since distances in music are expressed in terms of intervals, we can say that two musical structures are homometric if they have the same distributions of intervals without necessarily being equivalent up to elementary musical transformations such as transpositions (i.e., translations) or inversions (i.e., axial symmetries) \cite{mandereau2011b,mandereau2011a}.

The deep connection between tiling and homometry comes from the observation that considering a cyclic group $\Z_n$ and its decomposition into two non-periodic factors (i.e., two subsets of periodicity equal to~$n$), the tiling process constitutes a special case of an open conjecture in mathematics dating from the 1970s, namely Fuglede's spectral conjecture.

\subsubsection[A digression on Fuglede's spectral conjecture and its application to the discrete integer line]{A digression on Fuglede's spectral conjecture \\ and its application to the discrete integer line}

Interestingly, by going through the wide literature devoted to group factorization in the perspective of tiling problems in geometry, there is no evidence that all these problems can be related to the functional analysis domain and, more specifically, to an analytical problem concerning the commutativity of self-adjoint operators in the space $L^2(\Omega)$ of square-integrable functions defined on a domain $\Omega \subseteq \R^n$. This problem was raised by Bent Fuglede \cite{Fuglede1974} who was interested in the link between the spectrality of a domain $\Omega$ and the fact that it can tile the \textit{n}-dimensional Euclidean space by translation. More formally, a domain $\Omega \subseteq \R^n$ is \textit{spectral} if every function $f\in L^2(\Omega)$ can be represented in the following way: $f(x)=\sum f_k {\rm e}^{2\pi {\rm i} \lambda_k\cdot x}$, where $(\lambda_k)_{k\in \Z}$ is a~family of vectors and the exponentials ${\rm e}^{2\pi {\rm i}\lambda_k\cdot x}$ are mutually orthogonal maps.

In a more informal way, a domain $\Omega$ is spectral if it admits an orthonormal basis of complex exponentials, i.e., a Fourier decomposition. Fuglede's spectral conjecture can be stated in the following way.

\begin{con}[Fuglede]\label{con1}
A domain $\Omega \subset \R^n$ is spectral if and only if it tiles $\R^n$ by translation.
\end{con}

Contributions by several mathematicians including Fields medalist Terence Tao \cite{Tao2004}, as well as Mate Matolcsi and Mihalis Kolountzakis \cite{Kolountzakis2010}, have shown that the conjecture is false in all dimension $n>2$. In contrast to what happens in convex domains, where the conjecture has very recently been proved true in all dimensions \cite{Lev2022}, it is therefore possible to find domains $\Omega$ which tile $\R^n$ with $n>2$ without admitting a spectrum. Surprisingly, the conjecture is still open in dimension 1 and 2. The case $n=1$, once restricted to the integer $\Z$ line, enables to make a fruitful link between the spectral conjecture and the tiling rhythmic canons. In order to understand this connection, we need to show that tiling rhythmic canons can be expressed in a natural way in spectral terms via (the set of the zeroes of) the discrete Fourier transform associated to a given subset $A\subseteq \Z$. First of all, note that, since any given finite subset $A\subseteq \Z$ tiling the integer line by translation always tiles with a given period,\footnote{This result has been shown independently by Nicolaas Govert de Bruijn \cite{Bruijn1950} and Gy\"{o}rgy Haj\'os \cite{Hajos1950} in the 1950s.} we are in fact factorizing~$\Z_n$ into a direct sum of two subsets, one of which is the set $A$. Formally, it exists a set $B$ such that $\Z_n =A\oplus B$. We can therefore define the discrete Fourier transform associated to a given subset~$A\subseteq \Z_n$ without loss of generality on the tiling process.

\begin{defi}
Let $A\subseteq \Z_n$ be a subset of $\Z_n$, then the discrete Fourier transform $F_A$ is the discrete Fourier transform of its characteristic function, i.e., the map from $\Z_n$ to $\C$ which sends every element $t\in \Z_n$ to \smash{$\sum {\rm e}^{\frac{-2\pi {\rm i}kt}{n}}$}, where the sum is done for all elements $k\in A$.
\end{defi}

We will denote with $Z_A$ the set $Z_A =\{t\in \Z_n \mid F_A(t)=0\}$ of the zeroes of $F_A$. We recall the following property which expresses the factorization of a cyclic group in terms of the zeroes of the Fourier transforms of the respective factors (see~\cite{Amiot2009} for a proof):

\begin{pro}
Let $A$, $B$ be two subsets of $\Z_n$. Then $\Z_n =A\oplus B$ if and only if $Z_A \cup Z_B = \Z_n \backslash \{0\}$ and $\#A \times \#B=n.$
\end{pro}

Fuglede's conjecture (in dimension 1 and restricted to the discrete case) turns out to be naturally linked to homometry by virtue of the following observation: if a rhythmic pattern tiles the musical line by translation (i.e., it generates a tiling rhythmic canon), so does any rhythmic pattern that is homometric to the initial one. Moreover, from a spectral conjecture perspective, one only has to consider tiling canons associated to factorizations of a cyclic group as a direct sum of two non-periodic subsets (i.e., Vuza canons), since, as it has been shown by Emmanuel Amiot, all the other canons verify Fuglede's conjecture \cite{Amiot2016}. The interested reader can refer to Greta Lanzarotto's Ph.D.~Thesis \cite{Lanzarotto2022} for the most recent theoretical and computational account of Vuza Canons as potential candidates for approaching Fuglede's spectral conjecture as well as other open problems, such as Coven and Meyerowitz T2 conjecture.\footnote{See \cite{Coven1999}.} For a description of tiling canons as a key to approach open mathematical conjectures, see the survey \cite{Andreatta2015} as well as the special issue of \textit{J.~Math. Music} entitled ``Tiling problems in music"~\cite{Andreatta2009}.

\subsubsection{A special case of homometric structures: Babbitt's theorem}\label{sub:babbitt}
Milton Babbitt's hexachord theorem is surely one of the most celebrated results in ``mathemusical'' research. It expresses a property of invariance of the intervallic content of a chord with respect to its complement. Cyclic groups are traditionally used to represent musical structures such as chords, melodies or rhythms. In particular, a hexachord is a subset of $6$ notes over the~$12$ of the chromatic scale \[\Z/12\Z\equiv\{C, C\#, D,\ldots, B\}.\] Babbitt realized that the same \emph{intervals} appear with the same \emph{multiplicity} in the complementary hexachord $A^c$ as in the hexachord $A$. The hexachordal theorem simply states that $A$ and $A^c$ of measure $1/2$ are homometric. Note that this is a special case of a property -- known as Z-relation in the musical set-theoretical tradition~\cite{Forte1973} -- which can hold for sets $A$ and $B$ of equal cardinality and having the same interval content without being each other's complement. More generally, given a $n$-tone equal temperament which divides the octave into $n$ equal parts, it is possible to find $m$-tuples of sets of equal cardinality which are homometric. In the quarter-tone equal temperament, for example, there are three 12-tuples of sets of cardinality equal to 12. This is at the present the highest value of homometric sets that have been obtained by computational methods \cite{Wang2023}. Surprisingly, although many composers made use either consciously or unconsciously of Z-relation in their composition \cite{Goyette2012}, it must be noted that the question of the perceptual relevance of homometry still remains unanswered even in the simple case of couples of Z-related chords in the twelve-tone equal temperament. Nevertheless, the classification of homometric structures of different cardinalities within a given cyclic group $\Z_n$, i.e., generalized Z-related sets -- is an interesting combinatorial problem that constitutes an active field of research in computational musicology and micro-tonal composition \cite{Jedrzejewski2013, Wang2023}.

Whereas some examples of homometric structures can be found in the research carried on in the 1940s in the field of crystallography (see \cite{buerger76, patterson44, rosenblatt82}), the American music-theorist and composer Milton Babbitt made systematic use of Z-related musical structures within his serial technique, in which he applied the hexachordal theorem to the twelve-tone equal temperament~$\Z/12\Z$~\cite{babbitt1955}. Since Babbitt's original formulation \cite{babbitt1955} and its first complete proof by Ralph Hartzler Fox \cite{fox1964}, the hexachordal theorem has been discussed, reproved and sometimes generalized by several authors including David Lewin \cite{lewin1959}, Howard J.~Wilcox \cite{wilcox1982}, Steven K.~Blau~\cite{blau99}, Daniele Ghisi \cite{ghisi2006}, Emmanuel Amiot \cite{amiot2006} and Brian J.~McCartin \cite{mccartin2016}. One also finds a full characterization of simple graphs exhibiting the hexachordal property in \cite{AG01} by T.A.~Althuis and F.~G\"{o}bel. The hexachordal property has also been studied by David Lewin in \cite{lewin1987} within the so-called \emph{transformational music theory} and in particular in the context of (transformation) groups $T$ acting on a musical space $S$ in a simply transitive way (the so-called \textit{generalized interval systems}).
The uniquely determined group element mapping $x$ to $y$ is called \emph{interval} and denoted by $\mathrm{Int}(x,y)$. By choosing $e\in S$ as a reference, we can identify $T$ with $S$ through \[x\in S\mapsto \mathrm{Int}(e,x)\in T.\] In this way, the group action of $\mathrm{Int}(x,y)$ is identified with the left product by $z\mapsto \big(yx^{-1}\big)z$. The triplet $(S,T,\mathrm{Int})$ was called \emph{generalized interval system} or \emph{GIS} by David Lewin. This is the language for the proof of the hexachordal theorem for locally compact groups presented in~\cite{mandereau2011b,mandereau2011a}.

In the remaining part of this article, we discuss the new general framework that we introduced in the short companion paper \cite{Andreatta2023}, where the statements and their proofs are given in the framework of metric spaces with a probability measure and we illustrate our generalized hexachordal theorems with original constructions and examples. It must be observed that the approach we describe in this paper does not entirely fit within the previous diagram describing what we called the ``mathemusical'' dynamics, since there is no need, in general, to go through a computer-aided model to connect the musical original problem and the final mathematical result. Moreover, these constructions remain, at the present, highly speculative and the generalized results still need to find the appropriate musical interpretation which is far beyond the scope of the present study.

\section{Generalizing Babbitt's hexachordal theorem}
In our recent contribution \cite[Theorem 1.3]{Andreatta2023}, we give a full characterization of the spaces that host the hexachordal property. There the word \emph{interval} is given a metric meaning similar to~\cite{AG01} and the \emph{multiplicity} is measured through a new probabilistic approach. The heart of the very short proof is the same as \cite{fox1964}. We also generalize our result to more symbolic spaces \cite[Theorems 4.2 and 4.5]{Andreatta2023}. These results are recalled in Section \ref{sub:recap}.

\subsection[Naive illustration of the hexachordal phenomenon and heuristic proof in the case of the sphere Sp\^{}2]{Naive illustration of the hexachordal phenomenon \\ and heuristic proof in the case of the sphere $\boldsymbol{\Sp^2}$}

Among the closest extensions of the original hexachordal theorem on the discrete circle $\Z/12\Z$ are the ones to the continuous circle $\Sp^1$ and to the spheres $\Sp^3$ and $\Sp^7$ \cite{ballinger2009,mandereau2011b, mandereau2011a}. Our recent paper \cite{Andreatta2023} adds the other spheres $\Sp^d$ to the list. Since it speaks a lot to the inhabitants of the Earth, in this paragraph we would like to naively illustrate the hexachordal property on $\Sp^2$ and suggest a statistical proof for it. Our discussion should also be designed for non mathematician readers. To describe a set $A\subset \Sp^2$, one can look at the mean distance between two points picked randomly from $A$, where we assume that $\Sp^2$ is equipped the surface measure $\mu$ and the chord distance $d$. By mean distance, we mean
\[M_1(A)=\mu(A)^{-2}\iint_{A\times A} d(x,y) \d\mu(x)\d\mu(y)\]
the value corresponding to $p=1$ in the range of the power mean distances $(M_p(A))_{p>0}$, where
\begin{equation*}
M_p(A) := \left(\frac1{\mu(A)^2}\iint_{A\times A}d(x,y)^p\, \d\mu(x)\d\mu(y)\right)^{1/p}.
\end{equation*}
It is clear that rotating $A$ on $\Sp^2$ does not modify $M_1(A)$. To state the obvious the other power mean distances -- as for instance the quadratic mean distance $M_2(A)$ -- are also conserved after rotation. Finally, the (essential) diameter of $A$ is obviously conserved -- besides the fact it is $\lim_{p\to \infty}M_p(A)=\sup_p M_p(A)$. Nothing surprising in all that: the set $A$ is ``the same'' before and after rotation.

As we have proved in \cite{Andreatta2023}, if $\mu(A)=\mu\big(\Sp^2\big)/2$ another -- this time nontrivial -- operation conserves $M_1(A)$ and any other power mean {$M_p(A)$}, namely the complementary map
\begin{align*}
 A\mapsto A^c:=\Sp^2\setminus A.
\end{align*}
This invariance by the complementary map is the signature of the hexachordal theorem. Recall from Section~\ref{sub:babbitt} that we are looking for invariance of the multiplicity of the intervals. The proper notion for this is the one of \emph{distribution} (or \emph{law}) of the random distance between two independent points. Recall that the law of a random variable $X\colon \Omega\to F$ from a probability space~$(\Omega,\mathcal{A},\P)$ to a measured space $(F,\mathcal{B})$ is a probability measure $\P^X$ on $F$ defined by~${\P^{X}=\P\circ X^{-1}}$. Concretely $\P^X(A)$ represents the probability for $X$ to be in $A$. By definition, it is \[\P\big(X^{-1}(A)\big)=\P(\{\omega\in \Omega\colon X(\omega)\in A\})\] and is usually denoted by $\P(X\in A)$. Our proof of this invariance in \cite{Andreatta2023} is very basic~-- see below for the heuristic~-- and is similar to the one of Fox~\cite{fox1964}. Note that the invariance of the distribution implies the one of the power moments. The reason why we first mentioned the invariance of $M_p$ is that for many readers mean distances are more intuitive than probability distribution. However, notice that in the special case of $\Sp^2$, where the diameter is finite both are equivalent since according to the theory of the \emph{Hausdorff moment problem} the distribution of the random distance is uniquely characterized by the power means sequence $(M_p(A))_{p\in \mathbb{N}}$. To give another concrete and, we think,illustration of what distribution invariance implies one can consider $A$ to be the set of points with latitude between $-30^\circ$ and $30^\circ$, as illustrated in Figure~\ref{fig:sphere} (in white). We have the following application: the probability that the random distance is smaller than $\sqrt{2}$ (the distance between the poles and the equator) is the same for $A$ and $A^c$. While with simple geometric considerations on the set $A^c$ one proves that this probability is~$1/2$, a direct computation on~$A$ may first appear out of reach.

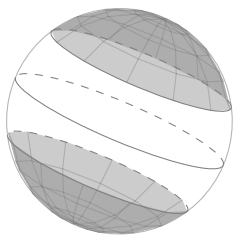
\begin{figure}[t]\centering
\def\r{1.5}
\tdplotsetmaincoords{80}{125}
\begin{tikzpicture}[tdplot_main_coords]
 \begin{scope}[thin,black!30,rotate=-23]
 \draw[tdplot_screen_coords,rotate=-23] (0,0,0) circle (\r);
 \end{scope}
 \begin{scope}[gray]
 \tdplotCsDrawLatCircle[rotate=-23]{\r}{30}
 \tdplotCsDrawLatCircle[rotate=-23]{\r}{-30}
 \tdplotCsDrawLatCircle[rotate=-23]{\r}{0}
 \end{scope}
\begin{scope}[opacity=0.4,rotate=-23]
\tdplotsetpolarplotrange{0}{60}{0}{360}
\tdplotsphericalsurfaceplot{24}{12}%
{\r}{gray}{gray}{}{}{}
\tdplotsetpolarplotrange{120}{180}{0}{360}
\tdplotsphericalsurfaceplot{24}{12}%
{\r}{gray}{gray}{}{}{}
\end{scope}
 \end{tikzpicture}
\caption{Two points randomly picked in the bright region of the sphere have distance distributed equally as the one between points picked in the dark region (made of two caps).}\label{fig:sphere}
\end{figure}

Before giving a heuristic proof of the $\Sp^2$ case, let us stress the relevance of generalizing the hexachord theorem to this specific case. As discussed in \cite{mandereau2011a}, the case of the sphere $\Sp^2$ is a~challenging one with respect to traditional music-theoretical constructions since David Lewin's notion of generalized interval systems (GIS) and related interval content are meaningless. In fact, as the authors recalled, unlike the continuous circle $\Sp^1$and the spheres $\Sp^3$ and $\Sp^7$, there is no group structure and a Haar measure which can be provided on the sphere $\Sp^2$ which will be compatible with its natural topology. Our general framework enables to overcome this difficulty but at the same time it asks for new musical interpretations that cannot be based anymore on well-known music-theoretical constructions such the GIS.

Let $(x_1,y_1),\ldots,(x_N,y_N)$ be a large sample of pairs of random points from $\Sp^2$.
Intuition (backed up by theorems such as the law of large numbers) tells us that we can estimate~$M_p(A)$ and $M_p(A^c)$ as well as other metric quantities depending on the random distance by sorting pairs~$(x_k,y_k)$ such that $x_k\in A$ and $y_k\in A$ on~the one side and pairs satisfying $(x_k,y_k)\in (A^c)^2$ on the other side. Note that the expected size of these subsamples is~$N/4$. We claimed that the distribution of the distance is conserved by $A\mapsto A^c$ so our two samples $S_1= \big\{d(x_k,y_k)\colon (x_k,y_k)\in A^2\big\}$ and $S_2=\big\{d(x_k,y_k)\colon (x_k,y_k)\in (A^c)^2\big\}$ should have the same statistical aspect when $N$ tends to infinity. The deep reason for this phenomenon is only revealed once a third sample is added to the two others simultaneously, namely $S_3=\{d(x_k,y_k)\colon (x_k,y_k)\in A\times A^c\}$. Doing this, the first sample $S_1$ becomes $S_1\cup S_3$ that collects the distances $d(x_k,y_k)$ with $x_k\in A$ and~$y_k\in A\cup A^c=\Sp^2$~-- i.e., there is no restriction on $y_k$~-- whereas $S_2$ becomes $S_2\cup S_3$, where the pairs $(x_k,y_k)$ are in $\Sp^2\times A^c$ -- again one point, here $x_k$, is free. Now it appears that in both cases we are considering the typical random distance to one given point of the sphere. The fact that this point, $x_k$ (respectively $y_k$) is in $A$ (respectively $A^c$) has no incidence on the random distance. Therefore, the two augmented samples have the same properties (up to variations due to the sampling) and since we added the same sample to both, so do the initial samples. This concludes our heuristic proof.

In \cite{Andreatta2023}, we directly implement this strategy with probability measures in place of random samples and metric spaces more general than $\Sp^2$. One can check above that the geometric property of $\Sp^2$ that was important for the proof and this scheme of proof is that the random distance to a fixed point does not depend of this point. It is stated in \cite[Definition 1.2]{Andreatta2023} in purely geometric terms through the constant volume condition (CVC).

 \begin{defi}[constant volume condition]\label{def:cgc} A metric measure space $(\X,d,\mu)$ is said to satisfy the \emph{constant volume condition} if there exists a function $\rho$ on $[0,\infty)$ such that for any center $x\in \X$ and radius $r\in[0,\infty)$ the closed ball $\mathcal{B}(x,r)=\{y\in \X\colon d(x,y)\leq r\}$ has measure $\rho(r)$. In an equivalent way,
 \begin{align}
\forall x,y\in \X, \ \forall r\geq 0, \qquad\mu(\mathcal{B}(x,r))=\mu(\mathcal{B}(y,r)).\tag{CVC}
 \end{align}
\end{defi}
In the following, we sometimes call \emph{CVC space} a metric probability space which satisfies the constant volume condition.

\subsection{Statement of three generalized hexachordal theorems}\label{sub:statements}\label{sub:recap}

We recall the statement of the three theorems we proved in \cite{Andreatta2023}.

\begin{them}[hexachordal theorem for metric probability spaces]\label{thm1}
Let $(\X,d,\mu)$ be a metric probability space. Assume that it satisfies the constant volume condition. Then for every Borel set $A$ of $\mu$-measure $1/2$, with notation $A^c=\X\setminus A$ one has
\begin{align}\label{eq:ht}
\mu^{2}\big\{(x,y)\in A^2\colon d(x,y)\in E \big\}=\mu^{2}\big\{(x,y)\in (A^c)^2\colon d(x,y)\in E \big\}\tag{Hex}
\end{align}
for every open subset $E\subset [0,\infty)$, where $\mu^2$ is the product measure $\mu\times \mu$ used for the $($measurable$)$ sets of pairs $(x,y)\in\X^2$.
\end{them}

From the perspective of homometry theory, as interpreted within this new probabilistic framework, two subsets $A$ and $B$ are homometric if for every $E$
\begin{align*}
\mu^{2}\big(\big\{(x,y)\in A^2\colon d(x,y)\in E\big\}\big)=\mu^{2}\big(\big\{(x,y)\in B^2\colon d(x,y)\in E\big\}\big).
\end{align*}
Therefore, Theorem \ref{thm1} stated that subsets $A$ of mass $1/2$ are homometric to $A^c$.
This definition generalizes to homometric probability measure spaces as follows. The spaces $(\X,d_1,\mu)$ and~$(\mathfrak{Y},d_2,\nu)$ are homometric if
\begin{align*}
\mu^{2}\big(\big\{(x,y)\in \X^2\colon d_1(x,y)\in E\big\}\big)=\nu^{2}\big(\big\{(x,y)\in \mathfrak{Y}^2\colon d_2(x,y)\in E\big\}\big)
\end{align*}
for every Borel $E\subset \R$.

Now we recall the statements of two more general theorems \cite[Theorems 4.2 and 4.5]{Andreatta2023}. We need to recall that a \emph{balanced decomposition} $(\mu_0,\mu_1)$ of $\mu$ is a pair of probability measures such that $\mu_0+\mu_1=2\mu$ (see \cite[Definition~4.1]{Andreatta2023}).

\begin{them}[characterization for metric probability spaces]\label{them2}
Let $(\X,d,\mu)$ be a metric probability space. The following properties are equivalent:
\begin{enumerate}\setlength{\leftskip}{0.46cm}\itemsep=0pt
\item[$({\rm CVC}')$]There exists a set $\X'\subseteq \X$ of full measure for $\mu$ such that the constant volume condition is satisfied on $(\X',d,\mu)$.
 \item[$({\rm Ind})$] For any independent random variables $X$ and $Y$ of law $\mu$ and $D=d(X,Y)$, the random variables $X$, $Y$ and $D$ are pairwise independent.
\item[$({\rm Hex}')$] For every balanced decomposition $(\mu_0,\mu_1)$ of $\mu$ and two random triples $(X_i,Y_i,D_i)_{i=0,1}$, where for every $i$, $(X_i,Y_i)$ is a pair of independent random variables of law $\mu_i$ and $D_i=d(X_i,Y_i)$, we have the equality on distributions $\P(D_0\in \cdot)=\P(D_1\in \cdot)$.
\end{enumerate}
\end{them}

\begin{them}[characterization for abstract probability spaces]\label{them3}
Let $(\X,\mathcal{F},\mu)$ be a probability space and $f$ a measurable symmetric function into a measured space $(\mathcal{M},\mathfrak{M})$. The following properties are equivalent:
\begin{enumerate}\setlength{\leftskip}{0.38cm}\itemsep=0pt
\item[$({\rm Ind})$] For any independent random variables $X$ and $Y$ of law $\mu$ and $F=f(X,Y)$, the random variables $X$, $Y$ and $F$ are pairwise independent.
\item[$({\rm Hex}')$] For every balanced decomposition $(\mu_0,\mu_1)$, considering the triples $(X_0,Y_0,F_0)$ and \linebreak $(X_1,Y_1, F_1)$, where for $i=0,1$ the pair $(X_i,Y_i)$ is made of independent random variables of law~$\mu_i$ and $F_i=f(X_i,Y_i)$, we have equality of both distributions, $\P(F_0\in \cdot)=\P(F_1\in \cdot)$ as measures on $\mathcal{M}$.
\item[$({\rm Hex}'')$] For any balanced decompositions $(\mu_0,\mu_1)$ and $(\nu_0,\nu_1)$, where for $i=0,1$, $X_i$ has law $\mu_i$, $Y_i$ has law $\nu_i$ and $F_i=f(X_i,Y_i)$, we have equality of both distributions $\P(F_0\in \cdot)= \P(F_1\in \cdot)$.
\end{enumerate}
Moreover, if $f$ is no longer supposed to be symmetric $({\rm Ind})\Leftrightarrow({\rm Hex}'')$ still holds as well as $({\rm Hex}'')\Rightarrow({\rm Hex}')$.
\end{them}

The next section is related with the two first generalized hexachordal theorems since it is devoted to the CVC spaces. Section \ref{sec:gene} is about Theorem \ref{them3}. In this case, the ``interval'' func\-tion~$f$ is not necessary a distance but may for instance be $(x,y)\mapsto x^{-1}\cdot y$ in the case of a~group.

\section{Metric probability spaces satisfying (CVC)}\label{sec:examples}

In this section, we provide examples of spaces satisfying the constant volume condition whose definition is recalled in Definition \ref{def:cgc}. Recall that it is a sufficient and almost necessary condition for the hexachordal theorem (Theorems \ref{thm1} and \ref{them2}). In Section \ref{sub:trans}, we report on some examples obtained from the simple fact that transitive spaces satisfies the CVC. See also Remark~\ref{rem:nontran} for the intriguing nontransitive spaces with CVC. In Section \ref{sub:constr}, we report some easy constructions to create new CVC spaces from established CVC spaces.

\subsection{Discrete and continuous CVC spaces}\label{sub:trans}

A classical group-theoretical framework which ensures (CVC) is the one of transitive group actions. We consider the metric measure spaces $(\X,d,\mu)$ such that for every $x$ and $y$ in $\X$ there exists a map $f\colon \X\to \X$ that satisfies
\begin{itemize}\itemsep=0pt
\item $f(x)=y$,
\item $d(f(z),f(z'))=d(z,z')$ for every $z,z'\in \X$, so that $f$ is an isometry,
\item $f_\#\mu=\mu$, where $f_\#\mu:=\mu\big(f^{-1}(\cdot)\big)$ is the law of $f(X)$ if $X$ has law $\mu$.
\end{itemize}
The constant volume condition follows:
\[\rho_{x}(r)=\mu(\mathcal{B}(x,r))=\mu\big(f^{-1}\mathcal{B}(f(x),r)\big)=\mu(\mathcal{B}(f(x),r))=\rho_{y}(r).\]

\begin{defi}[transitive space]
We call the metric measure space $(\X,d,\mu)$ a \emph{transitive space} if the group $G$ of maps $f$ as above (isometries preserving the measure) acts transitively on $\X$, i.e., such that for every $x,y\in \X$ there exists $f\in G$ such that $f(x)=y$.
\end{defi}

Let us first focus on the class of transitive examples satisfying (CVC) in the discrete setting of finite graphs with their counting measure. Note that in this case the condition $f_\#\mu=\mu$ is automatically satisfied because $f$ is one-to-one. Graphs satisfying the two first conditions are usually called (vertex-)transitive graphs and $f$ a graph isomorphism. Many of these graphs are the Cayley graph of a finite group. Recall that if a group $\X$ is generated by a finite system of generators $\Sigma\subseteq \X$, the Cayley graph attached to $(\X,\Sigma)$ is the graph with vertices $\X$ and edges the pairs $(x,y)$ such that $x^{-1}y\in \Sigma$ or $y^{-1}x\in \Sigma$. As usual, we denote this adjacency relation by~$x\sim y$. Let us check that these spaces $\X$ with the counting measure and the path distance are of transitive type and hence satisfy the constant volume condition. Given $x$ and $y$ we choose for $f$ the translation defined by $\tau_{v}\colon z\mapsto v z$, where $v=yx^{-1}$, so that $f(x)=y$. It is an isometry because $\tau_v(z)\sim \tau_v(z')$ if and only if $z\sim z'$, which follows from $\tau_v(z)^{-1}(\tau_v(z'))=z^{-1}z'$. Finally, recall that it preserves the counting measure since it is one-to-one. Basic examples of this type (i.e., finite Cayley graphs) are
\begin{itemize}\itemsep=0pt
\item The symmetric group $\mathfrak{S}(n)$ with, for instance, for $\Sigma$ the set of transpositions.
\item The group $(\Z/n_1\Z)\times \cdots\times (\Z/n_k\Z)$ with $\Sigma=\{(0,\ldots 0,\pm1,0,\ldots,0)\}$. Note that for $n_1=\cdots=n_k=2$, we find the so-called hypercube $\{0,1\}^k$. For $k=1$ and $n_1=12$, we recover the chromatic scale $\Z/12\Z$.
\end{itemize}
Note that a vertex-transitive graph may not be the Cayley graph attached to some $(\X,\Sigma)$, a counterexample being the \emph{Petersen graph} (a famous graph with 10 vertices and 15 edges) another one the graphs made of the edges and vertices of the \emph{dodecahedron}, \emph{icosahedron} and the \emph{truncated icosahedron}

We list now some continuous transitive examples among the most basic:
\begin{itemize}\itemsep=0pt
\item The (hyper)torus $\mathbf{T}^d=\Sp^1\times\cdots\times \Sp^1$ of dimension $d$ with its normalized volume (among other tori).
\item Any sphere $\Sp^d$ or product of spheres with their normalized volumes.
\item 
The Klein bottle. When the Euclidean space $\R^2$ is made a quotient through the group spanned by the translation $(x,y)\mapsto (x+1,y)$ and the glide reflexion $(x,y)\mapsto (1-x,y+1)$\footnote{The elements of this group are the isometries of the form $(x,y)\mapsto (k\pm x, y+l)$, where $\pm$ is $+$ if and only if $l$ is even.} the translations of $\R^2$ remains isometries that are acting transitively. Topologically the quotient space a Klein bottle with fundamental domain the square $[0,1)\times [0,1)$ (the lower and upper sides are identified after inversion of the orientation). Any Klein bottle of volume 1 obtained in a similar way will be transitive, satisfy the constant volume condition and hence host the hexachordal property. Note that two such Klein bottles are generally not isometric.
\item For more exotic examples, we can think to Albanese tori. Their topology is different from the one of usual tori.
\end{itemize}

\begin{rem}[nontransitive CVC spaces]\label{rem:nontran}
As shown in \cite{AG01,Andreatta2023}, there also exists nontransitive CVC spaces that are graphs. It is unknown whether there exists a nontransitive Riemannian manifold of finite measure with the CVC. We proved in \cite[Section~3]{Andreatta2023} that it does not exist for surfaces. Note also that considering open submanifolds with infinite Riemannian volume and at most one singularity it has been shown in \cite{KP} that the cone $C=\big\{x\in \R^4\colon x_4^2=x_1^2+x_2^2+x_3^2\big\}$ with the \big(Euclidean induced by $\R^4$\big) chord distance $d_{\R^4}$ and the induced volume (of dimension~3)~$\mu$ satisfies CVC. This space $\big(C,d_{\R^4}\,\mu\big)$ is nontransitive: it possesses for instance a singularity at~0. Moreover, the volume of the balls of given radius is the same as for $\R^3$, so that $C$ and~$\R^3$ can be qualified homometric in a general sense. A uniqueness result for $C$ among subsets of Euclidean spaces is also established in \cite{KP}.
\end{rem}

\subsection{Constructions with CVC spaces}\label{sub:constr}

In the following, we denote by $\rho$ the \emph{$($constant$)$ volume function} of a metric space $(\X,d,\mu)$ that satisfies (CVC). It is defined by
\[\rho(r)=\mu(\mathcal{B}(x_0,r)),\]
where $\mathcal{B}(x,r)$ is the closed ball of radius $r\geq 0$ and center $x$. Here $x_0$ is some (or, due to CVC, any) point of $\X$.

In the case of probability spaces, $\rho$ is the cumulative distribution of $D=d(X,Y)$, where $X$ and $Y$ are independent of law $\mu$, i.e.,
\begin{align*}
\rho(r)&=\P(d(X,Y)\leq r) =\iint \1(d(x,y)\leq r)\mu\otimes\mu({\rm d}x,{\rm d}y) \\
 &=\mu^{2}\big(\big\{(x,y)\in \X^2\colon d(x,y)\leq r\big\}\big)
\end{align*}
even though $(\X,d,\mu)$ does not satisfies the CVC. Since cumulative distribution functions characterize the distribution of a random variable, two metric probability spaces are homometric if and only if they have the same volume function. Note, for instance, that Example~3.2 in \cite{Andreatta2023} and~$\Z/7\Z$ with the relation $(x\sim y \Longleftrightarrow y-x\in \{-1,1,-3,3\})$ have the same volume function but are not isometric: the first is nontransitive and the second is transitive.

\begin{ex}[products]\label{ex:products}
Let $(\X_1,d_1,\mu_1)$ and $(\X_2,d_2,\mu_2)$ be two spaces satisfying (CVC) with volume functions $\rho_1$ and $\rho_2$, respectively. Then the product space $\X:=\X_1\times \X_2$ also satisfies it, with the product measure $\mu:=\mu_1\times \mu_2$.
Several choices are possible to combine the distances. We focus on the $\ell^p$ norms of $(d_1,d_2)\in\R^2$, where $p\in[1,\infty]$:
\begin{itemize}\itemsep=0pt
\item[$\ell^\infty$] We can set $d((x_1,y_1),(x_2,y_2))=\max(d_1(x_1,y_1),d_2(x_2,y_2))$. Let us comment on the example of the product of two finite graphs with their path distances. The resulting space is~$\X_1\times \X_2$ with the path distance resulting of the so-called \emph{strong product} of the two graphs. In fact, $d((x_1,y_1),(x_2,y_2))\leq 1$ if and only if $d((x_1,y_1)\leq 1$ and $d((x_2,y_2)\leq 1$. Denoting by $(x_1,y_1)\simeq(x_2,y_2)$ the relation $\{(x_1,y_1)\sim(x_2,y_2)$ or $(x_1,y_1)=(x_2,y_2)\}$, it follows that $(x_1,y_1)\simeq(x_2,y_2)$ if and only if $x_1\simeq x_2$ and $y_1\simeq y_2$. One can check that (CVC) is satisfied for $\rho=\rho_1\times \rho_2$.

\item[$\ell^p$] We can set $d^p((x_1,y_1),(x_2,y_2))=d_1(x_1,y_1)^p+d_2(x_2,y_2)^p$ and obtain for this choice the constant volume function $\rho(r)=\int_{[0,r]}\rho_2\big((r^p-t^p)^{1/p}\big)d\rho_1(t)$.
\item[$\ell^1$] The product of two graphs is the so-called \emph{Cartesian product} for which $(x_1,y_1)\sim(x_2,y_2)$ if and only if ($x_1=x_2$ and $y_1\sim y_2$) or ($y_1=y_2$ and $x_1\sim x_2$).
\item[$\ell^2$] If $\X_1$ and $\X_2$ are isometrically embedded in Euclidean spaces, so is the product with the~$\ell^2$ distance. For instance, the hexachordal phenomenon can be observed on $\Sp^1\times \{0,1\}\subseteq \R^2\times \R=\R^3$.
\end{itemize}
This example inductively extends to product spaces $\X:=\X_1\times\cdots \times\X_n$ and the product measure $\mu_1\otimes\cdots\otimes \mu_n$. This corresponds to the setting of independent random variables ${X=(X_1,\ldots,X_n)}$ and $Y=(Y_1,\ldots,Y_n)$ that we can moreover equip with the Hamming distance $D=d(X,Y)=\sum_{i=1}^n \1(X_i \neq Y_i)\times a_i$ with $a_i>0$, $i=1,\ldots,n$. For an illustration of Theorem~\ref{thm1}, one can, for instance, consider the uniform head and tail space $\X=\{h,t\}^5$ and apply to $A$ ``at least three consecutive tosses are the same'' since it has probability 1/2, as one can check. (Moreover, one can check that $A$ and $A^c$ are generically not isometric.)
\end{ex}

\begin{ex}[union of two spaces with the same constant volume function]\label{ex:twocgc}
We consider for $i=1,2$ two spaces $(\X,d_i,\mu_i)$ satisfying the constant volume condition for the same function $\rho$. We assume moreover that the two spaces are bounded. As noticed in Remark~\ref{rem:nontran}, they can be different. Define $\X$ the disjoint union $\X_1 \sqcup \X_2$ with probability measure $\mu=(1/2)(\mu_1+\mu_2)$ and distance defined by
\[d(x,y)=
\begin{cases}
d_i(x,y) &\text{if }x,y\in \X_i\text{ for some }i,\\
L&\text{otherwise.}
\end{cases}
\]
The constant volume condition and $({\rm Hex})$ are satisfied for any $L\geq 0$. This makes sense even though $d$ is not a distance (see Theorem \ref{them3}) but in order to save the triangle inequality we have to require that for every $i=1,2$ the distance between any two points of $\X_i$ is less than $2L$.
\end{ex}

\begin{ex}[graphs whose points are replaced by metric spaces]\label{ex:graphs}
Let $(G,d_0)$ be a finite graph with the constant volume condition for the counting measure. We scale it so that adjacent points have distance $L$. We replace $G=\bigcup_{i=1}^N\{x_1\}$ by $\X=\bigsqcup_{i=1}^N \X_i$ a family of metric spaces $(\X_i,d_i,\mu_i)$ with diameter smaller than $2L$ and satisfying (CVC) with moreover the same constant volume function, $\rho_1=\cdots=\rho_N$. On $\X$ we set $d(x,y)=d_i(x,y)$ if $x,y\in \X_i$ and $d(x,y)=d_0(x_i,x_j)$ if~$x\in \X_i$, $y\in \X_j$, $i\neq j$. It can be checked that the resulting space satisfies the constant volume condition. A special case is Example~\ref{ex:twocgc}.
\end{ex}

\begin{ex}\label{ex:morecgc}
We have indicated examples of metric probability spaces $(\X,d,\mu)$ satisfying (CVC) such that $D$ is an absolutely continuous or a discrete random variable. With $\{0,1\}\times \Sp^1$ in Example \ref{ex:products} if one takes $\ell_\infty$ and the Examples \ref{ex:twocgc} and \ref{ex:graphs} we see the possibility for $D$ to have both a~nontrivial atomic and absolutely continuous part. We introduce now the situation of a~space with the (CVC) such that $D$ is diffuse but not absolutely continuous. Precisely its law is the Cantor law and its cumulative distribution function $\rho$ is the Devil's staircase. For $\X$ we take the sequences $x=(x_1, x_2,\ldots)$ with $x_i\in \{0,1\}$ for every $i\geq1$. The measure is the one of head/tail model, i.e., we weight each digit with 1/2 independently. For the distance between~$x$ and $y$, we set $d(x,y)=\sum_{i=1}^\infty |x_i-y_i|(2/3^i)$. Note that it corresponds to a $\ell^1$ distance on an infinite product $\{0,1\}^{\N^*}$ weighted by a scaling factor $(2/3^i)$ on the $i$-th coordinate.
\end{ex}

\section{Remarks on the theorems}\label{sec:gene}

\subsection[Remark on weakening mu(A)=1/2 in Theorem 2.2]{Remark on weakening $\boldsymbol{\mu(A)=1/2}$ in Theorem \ref{thm1}}

Generalized versions of Babbitt's hexachordal theorem have been given, where the subset of interest and its complementary do not have the same size~\cite{mccartin2016}.
For Theorem \ref{thm1}, this translates into $A$ not having $\mu$-measure $1/2$ (and hence neither has $A^c$).
The principle of considering~$A\times A^c$ is still useful, but does not simplify as well in this more general case.
Namely, for any $r\in[0,+\infty]$, we have
\begin{equation*}
 \mu^{2}\big\{(x,y)\in A^2\colon d(x,y)\leq r \big\} - \mu^{2}\big\{(x,y)\in (A^c)^2\colon d(x,y)\leq r \big\} = \rho(r)[\mu(A)-\mu(A^c)].
\end{equation*}
However, as the two sets $A^2$ and $(A^c)^2$ have different measures, this generalization does not carry over neatly to the probabilistic formulations of Theorems~\ref{them2} and~\ref{them3}. Finally, note that~$\{d(x,y)\leq r\}$ in the left-hand side can be replaced by $\{d(x,y)\in E\}$ if one also replaces~$\rho(r)$ by $\int_{0}^{\infty} \1_E(r) {\rm d}\rho(r)$ in the right-hand side.

\subsection[Remarks on (Ind) in Theorem 2.4]{Remarks on $\boldsymbol{({\rm Ind})}$ in Theorem \ref{them3}}
Whereas in Theorems \ref{thm1} and \ref{them2} the law of $(X,D)=(X,d(X,Y))$ equals the law of $(Y,D)=(Y,d(X,Y))$ and $X$ and $D$ are independent if and only if $Y$ and $D$ are, the fact that in Theorem~\ref{them3} the function $f$ does not have to be symmetric (see the two last lines) makes that $({\rm Ind})$ ``$X$, $Y$ and $F=f(X,Y)$ are pairwise independent'' is different from ``$X$ and $F$ are independent''. In this paragraph we comment on this difference.

To illustrate that the pairwise independence $({\rm Ind})$ is necessary we consider $f(x,y)=y$ and observe Theorem~\ref{them3} does not apply properly. If $X$ and $Y$ are independent, so are $X$ and~${F=f(X,Y)=Y}$. It is clear that $Y$ and $F=Y$ are not independent so that $({\rm Ind})$ is not satisfied. Moreover, for $X_0$ and $Y_0$ independent of law $\mu_0$ the law of $F_0$ is $\mu_0$ and for $X_1$ and~$Y_1$ independent of law $\mu_1$ the law of $F_1=f(X_1,Y_1)$ is $\mu_1\neq \mu_0$. Therefore, $({\rm Hex}')$ is false and $({\rm Hex}'')$ is also false for $(\nu_0,\nu_1)=(\mu_0,\mu_1)$.

We continue with a corollary of the second case of Theorem~\ref{them3} (not necessary symmetric functions, see the last two lines). It applies, in particular, to topological groups with a left-invariant probability measure through the function $f(x,y)=x^{-1}\cdot y$, which is important with respect to the existing literature. See \cite[Corollary 4.6]{Andreatta2023} for a direct proof.

\begin{cor}[weakening $({\rm Ind})$ in Theorem \ref{them3} for antisymmetric functions $f$]\label{cor:anti}
Let $(\X,\mathcal{F},\mu)$ be a probability space and $f$ a measurable function defined on $\X\times \X$ with values in a measurable space $(\M,\mathfrak{M})$. Let us assume that $f$ is antisymmetric in the sense there exists a measurable involution $i$ on $\M$, i.e., a function $i\colon \M\to \M$ such that $i\circ i(m)=m$, for every $m\in \M$ with $f(x,y)=i\circ f(y,x)$.

Let $X$ and $Y$ be independent random variables of law $\mu$. If $F=f(X,Y)$ and $Y$ are independent, then $X$ and $F$ are also independent so that $({\rm Ind})$ is satisfied and $({\rm Ind})\Rightarrow({\rm Hex}')$ and $({\rm Ind})\Leftrightarrow({\rm Hex}'')$ apply in Theorem {\rm \ref{them3}}.

Similarly if $X$ and $Y$ as well as $X$ and $F$ are independent so are $F$ and $Y$ and $({\rm Ind})$ is satisfied.
\end{cor}
\begin{proof}
Let us assume that $(X,Y)$ and $(Y,F)$ are pairs of independent random variables. Thus~$Y$ and $i(F))=i\circ f(X,Y)=f(Y,X))$ are independent random variables. But the law of $(Y,f(Y,X))$ is the law of $(X,f(X,Y))$. Therefore, $X$ and $F$ are independent. The same argument to prove $({\rm Ind})$ also works if $X$ and $F$ are independent.
\end{proof}

We now concretely illustrate Theorem \ref{them3} and Corollary \ref{cor:anti} with $f(x,y)=x^{-1}\cdot y$ in the special case of some $A$ of cardinality 6 in the group $\Z/3\Z\times \Z/4\Z$ of cardinality 12. We also give the invariant distribution of the distance as an illustration of Theorems \ref{thm1} and \ref{them2}.

\begin{ex}
We consider $G=\Z/3Z\times \Z/4\Z$ with $f\colon (x,y)\mapsto x^{-1}y$. For
\[A=\{ (1,0),(1,2),(2,0),(2,1),(2,2),(2,3)\}\]
and its complementary set, the random variables $F_i$, $i=0,1$ are distributed as follows:
\begin{center}\renewcommand{\arraystretch}{1.2}
\begin{tabular}{|c|c|c|c|c|c|}
\hline
$v$ &$(0, 0)$&(0,2) & $(0,1)$ and $(0,3)$&$(1,0)$ and $(2,0)$ \\
\hline
$\P(F_i=v)$ & $6/36$ & 6/36& 4/36& 2/36 \\
\hline
$v$ & \multicolumn{2}{c|}{$(1,1)$ and $(2,3)$}& $(1,2)$ and $(2,2)$& $(1,3)$ and $(2,1)$ \\
\hline
$\P(F_i=v)$ & \multicolumn{2}{c|}{3/36} & 2/36& 2/36 \\
\hline
\end{tabular}.
\end{center}
For the set of generators $\{(\pm 1,0),(0,\pm 1)\}$, this corresponds to the following distribution of the distances $D_i$, $i=0,1$:
\begin{center}\renewcommand{\arraystretch}{1.2}
\begin{tabular}{|c|c|c|c|c|}
\hline
$r$ & $0$ & $1$ & $2$& $3$\\
\hline
$\P(D_i=r)$ & $6/36$ & 12/36& 10/36& 8/36\\
\hline
\end{tabular}.
\end{center}
\end{ex}

Beside $f(x,y)=x^{-1}\cdot y$ the most popular example of function for which Theorem \ref{them3} applies is probably $f(x,y)=x\cdot y$ in the case of groups with a left-invariant probability measure. This was explained in \cite[Corollary~4.6]{Andreatta2023}. Here we state it again for finite groups (where the uniform measure is invariant) since it will be useful in the following subsections.

\begin{ex}[Caley tables]\label{ex:cayley}
Observe that $({\rm Ind})$ in Theorem \ref{them3} is satisfied for finite groups~$(\!G,\!\cdot)$ with their uniform measure $\mu$ and the function $f(x,y)=x\cdot y$. In fact, for $X$, $Y$ independent of law $\mu$ and every $(x,y)\in G^2$,
\begin{align*}
\P(X=x,\, F=y)=\P\big(X=x, Y=x^{-1}\!\cdot y\big)=\P(X=x)\P\big(Y=x^{-1}\!\cdot y\big)=\P(X=x)\!\times\! \frac1{\#G}.
\end{align*}
Therefore, $F$ is uniform on $G$ and independent from $X$. One proves in the same way that $F$ is independent from $Y$.

In fact, the value of $(X,Y,F)$ corresponds to a uniform intersection among $(\#G)^2$ possibilities in the Cayley tables, i.e., the multiplication table for the group $G$, where $X$ is the row and $Y$ the column. Therefore, this type of tables permits one to visualize the hexachordal property. See, e.g., \cite{fox1964,wilcox1982} where this fact is commented.
\end{ex}

\subsection[Remark on (Ind) Rightarrow (Hex') in Theorem 2.4]{Remark on $\boldsymbol{({\rm Ind})\Rightarrow({\rm Hex}')}$ in Theorem \ref{them3}}

Let us show that in Theorem \ref{them3} the implication $({\rm Hex}')\Rightarrow({\rm Ind})$ is false. To see this, let~$\X$ be the space $\{\star,\#,\S,\bullet\}$ equipped with the uniform measure $\mu$. Let the values $f(X,Y)$ of a~nonsymmetric function $f$ be given in Table~\ref{fig:tab}, where $X$ is the uniform choice of a row, $Y$ of a column and $F$ is the intersection. Our example is the right table of Table~\ref{fig:tab}, the left and the middle being parts of the explanation. Observe that the law of $F=f(X,Y)$ conditioned on the choice of a~row (the value of $X$) is not constant. Hence, $({\rm Ind})$ is not satisfied. The left table corresponds to a Cayley table as in Example~\ref{ex:cayley} so that $({\rm Hex}'')$ is satisfies. In the middle table below, we have swapped the two last rows. We can notice that the function is~no~longer symmetric. However, we stress that $({\rm Ind})$, $({\rm Hex}'')$ and $({\rm Hex}')$ remain. In $({\rm Hex}')$, given a~balanced decomposition~$(\mu_0,\mu_1)$ we have for $\P((X_i,Y_i)=(a,b))=\P((X_i,Y_i)=(b,a))$ for any pair $(a,b)$ and $i=0,1$. Therefore, we conserve $({\rm Hex}')$ if we swap the values of $f(a,b)$ and~$f(b,a)$. This is what we did for $(\S,\star)$ between the middle and the right table.
As commented above $({\rm Ind})$ (and $({\rm Hex}'')$) are no longer true after this operation.
\begin{center}

\begin{table}[h!]\centering
\begin{tabular}{c|cccc}
&$\star$&\#&\S&$\bullet$\\
\hline
$\star$&0&1&2&3\\
\#&1&2&3&0\\
\S&2&3&0&1\\
$\bullet$&3&0&1&2
\end{tabular}
\qquad
\begin{tabular}{c|cccc}
&$\star$&\#&\S&$\bullet$\\
\hline
$\star$&0&1&2&3\\
\#&1&2&3&0\\
\S&3&0&1&2\\
$\bullet$&2&3&0&1
\end{tabular}
\qquad
\begin{tabular}{c|cccc}
&$\star$&\#&\S&$\bullet$\\
\hline
$\star$&0&1&3&3\\
\#&1&2&3&0\\
\S&2&0&1&2\\
$\bullet$&2&3&0&1
\end{tabular}
\caption{The left and middle functions satisfy $({\rm Hex}')$ and $({\rm Hex}'')$. The right one satisfies $({\rm Hex}')$ but not $({\rm Hex}'')$.}\label{fig:tab}
\end{table}
\end{center}

\subsection{Nontransitive example where Theorem \ref{them3} applies}

In the metric setting, the constant volume condition is satisfied by many transitive metric probability spaces for which simple formulas exist. However, as recalled in Remark \ref{rem:nontran} (CVC) is not equivalent to being transitive. Although it is known since at least 1964 \cite{fox1964}, we would like to stress that exactly the same happens in the symbolic setting of Theorem \ref{them3}: not all spaces~$(\X,f,\mu)$ with (Ind) are isomorph to groups with the function $f(x,y)=x\cdot y$ or $x^{-1}\cdot y$. To take the Cayley table of a group as in Example \ref{ex:cayley} is only one way to have $X$, $Y$ and $F$ pairwise independent. In the discrete setting, it was already mentioned, for instance, in \cite{fox1964} that any Latin square gives rise to the hexachordal theorem. In Table~\ref{fig:tab3}, we represented such a Latin square with six symbols. `Latin' stands for the fact that each symbol appears once and only once on each row and column. A~rather simple argument in the theory of Latin squares/quasigroups tells that if the following table were a Cayley table it would also be the Cayley table of a group, where the neutral element would be both left and up in the upper and left header, respectively. Here it means that $\heartsuit$ could be considered as the neutral and the elements on the first row and column are the ones of the two headers. Notice that they appear in the same order: $\heartsuit$, $\Box$, $\triangle$, $\clubsuit$, $\diamondsuit$, $\spadesuit$. Thus the elements on the diagonal should be the squares. Since on the diagonal $\heartsuit$ only appears once, the group can neither be $\Z/6\Z$, $(\Z/2\Z\times \Z/3\Z)$ nor $\mathfrak{S}(3)$ for which there are at least two elements of order 2. But these groups precisely constitute the list of all groups of cardinality 6.

\begin{table}[h!]
\begin{center}
\begin{tabular}{|cccccc}
\hline
$\heartsuit$&$\Box$&$\triangle$&$\clubsuit$&$\diamondsuit$&$\spadesuit$\\
$\Box$&$\triangle$&$\diamondsuit$&$\spadesuit$&$\heartsuit$&$\clubsuit$\\
$\triangle$&$\clubsuit$&$\Box$&$\diamondsuit$&$\spadesuit$&$\heartsuit$\\
$\clubsuit$&$\spadesuit$&$\heartsuit$&$\Box$&$\triangle$&$\diamondsuit$\\
$\diamondsuit$&$\heartsuit$&$\spadesuit$&$\triangle$&$\clubsuit$&$\Box$\\
$\spadesuit$&$\diamondsuit$&$\clubsuit$&$\heartsuit$&$\Box$&$\triangle$\\
\end{tabular}
\caption{A Latin square that is not a Cayley table.}\label{fig:tab3}
\end{center}
\end{table}

\subsection{Remark on the Patterson function}\label{sub:pat}

A popular setting for Theorem \ref{them3} is the one of separable groups with a left invariant probability measure $\mu$. Recall, in particular, \cite[Corollary~4.6]{Andreatta2023} and Example~\ref{ex:cayley}. In this setting, some authors introduce the Patterson function of $A\subset \X$ defined by $\operatorname{Pat}_A\colon g\in \X \mapsto \mu(A\cap g\cdot A)$ and reformulate Babbitt's theorem as the functional equality $\operatorname{Pat}_A=\operatorname{Pat}_{A^c}$ that has to be satisfied for every $A$ of measure $1/2$. Let us explain that this is an equality of densities (with respect to~~$\mu$) that corresponds to our equality of measures $({\rm Hex})$ adapted to the case of general functions $f$. Note first that our formulation with measures is justified by the fact that in general $F=f(X,Y)$ does not possess a density with respect to $\mu$ (or to the Lebesgue measure, see Example~\ref{ex:morecgc}). Assume $F=X^{-1}\cdot Y$ and that the three components of $(X,Y,F)$ are pairwise independent of law $\mu$. We have the following rewriting of $({\rm Hex})$ for groups:
\begin{align}\label{eq:lr}
\P(F\in E \mid X\in A\text{ and }Y\in A)=\P(F\in E \mid X\in A^c\text{ and }Y\in A^c).
\end{align}
The left-hand side also writes $4\P(F\in E\text{ and }X\in A\text{ and }X\cdot F\in A)$. Since $F$ and $X$ are independent, this is
\begin{align*}
4\int_E\P(X\in A\text{ and }X\cdot g\in A) {\rm d}\mu(g)&=\int_E 4\operatorname{Pat}_A\big(g^{-1}\big) {\rm d}\mu(g).
\end{align*}
Using the left invariance of $\mu$ in the definition of $\operatorname{Pat}_A$, we get $\operatorname{Pat}_A\big(g^{-1}\big)=\operatorname{Pat}_A(g)$ and finally see that the law of $F$ conditional upon $(X\in A\text{ and }Y\in A)$ admits the density $4\operatorname{Pat}_A$ with respect to $\mu$. One can proceed identically for the right-hand side of \eqref{eq:lr} so that for every $E$ measurable
\[\P(F\in E \mid X\in A^c\text{ and }Y\in A^c)=\int_E 4\operatorname{Pat}_{A^c}(g) {\rm d}\mu(g).\]
 From \eqref{eq:lr}, it follows the equality of the two Patterson functions at almost every $g\in \X$. As proved in \cite{ballinger2009,buerger76,mandereau2011a,wilcox1982} in the adequate setting this identity in fact holds not only for almost every but for every $g\in \X$.

\subsection*{Acknowledgements}

The authors wish to thank the editors of this special issue for their invitation to celebrate Jean-Pierre Bourguignon's 75th birthday through a paper centered around some aspects of contemporary `mathemusical' research. We are particularly grateful to the anonymous referees for their careful reading of the manuscript and their valuable criticisms that helped us to improve many aspects of the present article. Many thanks to the colleagues at IRMA-University of Strasbourg and IRCAM-Sorbonne University for all the discussions that contributed to clarifying some of the ideas developed in this article. This research is partially supported by the Interdisciplinary Thematic Institute CREAA, as part of the ITI 2021-2028 program of the Universit\'e de Strasbourg, the CNRS, and the Inserm (funded by IdEx Unistra ANR-10-IDEX-0002, and by SFRI-STRAT'US ANR-20-SFRI-0012 under the French Investments for the Future Program).


\pdfbookmark[1]{References}{ref}
\LastPageEnding

\end{document}